\documentclass[11pt,reqno]{amsart}
\usepackage{amsfonts}
\usepackage{amsmath}
\usepackage{graphicx}
\usepackage{epsf}
\usepackage{amsthm}
\usepackage{amsfonts}
\usepackage{amssymb}
\usepackage{epsfig}
\usepackage{graphicx}
\usepackage{setspace}
\usepackage{color}
\usepackage{texdraw}
\usepackage{float}
\newcommand{\ith}[2]{{#1}_{\:\!{\mbox{\footnotesize{#2}}}}}
\newcommand{\itoj}[2]{\ovl{#1,\!#2}}
\newcommand{\vij}[3]{{#1}_{\:\!{#2}}\mbox{\footnotesize{[$#3$]}}}
\newcommand{\vj}[2]{{#1}\mbox{\footnotesize{[$#2$]}}}
\newcommand{\vi}[2]{{#1}_{\:\!#2}}
\def\dspl{\displaystyle}
\def\ovl{\overline}

\theoremstyle{definition}
\newtheorem{theorem}{Theorem}
\newtheorem{corollary}{Corollary}
\newtheorem{remark}{Remark}
\newtheorem{example}{Example}
\newtheorem{lemma}{Lemma}
\newtheorem{definition}{Definition}

\newcommand{\R}{\mathbb{R}}

\newcommand{\x}{\mathbf{x}}
\newcommand{\X}{\mathbf{X}}
\newcommand{\y}{\mathbf{y}}
\newcommand{\0}{\mathbf{0}}

\newcommand{\Bc}{\mathbf{c}}

\begin{document}

\author{Tsvetan Asamov}
\address[Tsvetan Asamov]{Department of Operations Research and Financial Engineering, Princeton University, 98 Charlton Street, Princeton, NJ 08540, USA}
\email{tasamov@princeton.edu}
\author{Adi Ben--Israel}
\address[Adi Ben--Israel]{Rutgers Business School,  Rutgers University, 100 Rockafeller Road, Piscataway, NJ 08854, USA}
\email{adi.benisrael@gmail.com}

\title[High Dimensional Clustering]
{A Probabilistic $\ell_1$ Method for Clustering High Dimensional Data}
\date{April 7, 2016}

\begin{abstract}
In general, the clustering problem is NP--hard, and global optimality cannot be established for non--trivial instances. For high--dimensional data, distance--based methods for clustering or classification face an additional difficulty, the unreliability of distances in very high--dimensional spaces.
We propose a probabilistic, distance--based, iterative method for clustering data in very high--dimensional space, using the $\ell_1$--metric that is less sensitive to high dimensionality than the Euclidean distance. For $K$ clusters in $\R^n$, the problem decomposes to $K$ problems coupled by probabilities, and an iteration reduces to finding $Kn$ weighted medians of points on a line. The complexity of the algorithm is linear in the dimension of the data space, and its performance was observed to improve significantly as the dimension increases.

\end{abstract}
\maketitle
\begin{spacing}{1.0}
\section{Introduction}
\label{sec:Introduction}
The emergence and growing applications of big data have underscored the need for efficient algorithms based on optimality principles, and scalable methods that can provide valuable insights at a reasonable computational cost.

In particular, problems with high--dimensional data have arisen in several scientific and technical areas (such as genetics \cite{kailing2004density}, medical imaging \cite{ye2010groupwise} and spatial databases \cite{kolatch2001clustering}, etc.)
These problems pose a special challenge because of the unreliability of distances in very high dimensions. In such problems it is often advantageous to use the $\ell_1$--metric which is less sensitive to the ``curse of dimensionality'' than the Euclidean distance.

We propose a new probabilistic distance--based method for clustering data in very high--dimensional spaces.
The method uses the $\ell_1$--distance, and computes the cluster centers using weighted medians of the given data points.
Our algorithm resembles well--known techniques such as fuzzy clustering \cite{BEZDEK-81} and $K$--means, and inverse distance interpolation \cite{SHEPARD-68}.

The cluster membership probabilities are derived from necessary optimality conditions for an approximate problem, and decompose a clustering problem with $K$ clusters in $\R^n$ into $Kn$ one--dimensional problems, which can be solved separately.
The algorithm features a straightforward implementation and a polynomial running time, in particular, its complexity is linear in the dimension $n$. In numerical experiments it outperformed several commonly used methods, with better results for higher dimensions.

While the cluster membership probabilities simplify our notation, and link our results to the theory of subjective probability, these probabilities are not needed by themselves, since they are given in terms of distances, that have to be computed at each iteration.

\subsection{Notation}
\label{subsec:Notation}
We use the abbreviation
$
\itoj{1}{K}:=\{1,2,\ldots,K\} 
$
for the indicated index set.
The $\ith{j}{th}$ component of a vector $\vi{\x}{i}\in\R^n$ is denoted $\vij{\x}{i}{j}$. The $\ell_p$--\textbf{norm} of a vector $\x=(\vj{\x}{j})\in\R^n$ is
\[
\vi{\|\x\|}{p}:=(\sum_{j=1}^n\,\big| \vj{\x}{j} \big|^p)^{1/p}
\]
and the associated $\ell_p$--\textbf{distance} between two vectors $\x$ and $\y$ is $d_p(\x,\y):=\|\x-\y\|_p$, in particular, the Euclidean distance with $p=2$, and the $\ell_1$--distance,
\begin{equation}
d_1(\x,\y)=\|\x-\y\|_1 = \sum_{j=1}^n\,\big|\vj{\x}{j}-\vj{\y}{j}\big|.
\label{eq:L1-distance}
\end{equation}
\subsection{The clustering problem}
\label{subsec:Clustering problem}
Given
 \begin{itemize}
 \item a set $\mathbf{X}=\{\vi{\x}{i}:\,i\in\itoj{1}{N}\}\subset\R^n$ of $N$ points $\vi{\x}{i}$ in $\R^n$,
 \item their \textbf{weights} $W=\{\vi{w}{i}>0:i\in\itoj{1}{N}\}$, and
 \item an integer $1\leq K \leq  N$,
 \end{itemize}
 partition $\mathbf{X}$ into $K$ \textbf{clusters} $\{\vi{\mathbf{X}}{k}:\,k\in\itoj{1}{K}\}$, defined as disjoint sets
where the points in each cluster are \textbf{similar} (in some sense), and points in different clusters are dissimilar. If by \textbf{similar} is meant \textbf{close} in some metric $d(\x,\y)$, we have a \textbf{metric} (or \textbf{distance based}) \textbf{clustering problem}, in particular $\ell_1$--\textbf{clustering} if the $\ell_1$--distance is used, \textbf{Euclidean clustering} for the $\ell_2$--distance, etc.
\subsection{Centers} In metric clustering each cluster has a representative point, or \textbf{center}, and distances to clusters are defined as the distances to their centers. The center $\vi{\Bc}{k}$ of cluster $\vi{\mathbf{X}}{k}$ is a point $\Bc$ that minimizes the sum of weighted distances to all points of the cluster,
\begin{equation}
\vi{\Bc}{k}:=\arg\min \,\big\{\sum_{\vi{\x}{i}\in\mathbf{X}_k}\,\vi{w}{i}\,d(\vi{\x}{i},\Bc)\big\}.
\label{eq:c^k}
\end{equation}
Thus, the metric clustering problem can be formulated as follows: 
Given $\mathbf{X}, W$ and $K$ as above, find centers $\{\vi{\Bc}{k}:\,k\in\itoj{1}{K}\}\subset\R^n$ so as to minimize
\begin{equation}
\min_{\Bc_1,\cdots,\Bc_K}\,\sum_{k=1}^K\,\,\sum_{\vi{\x}{i}\in \mathbf{X}_k}\,\vi{w}{i}\,d(\vi{\x}{i},\vi{\Bc}{k}),
\tag{\textbf{L}.$K$}
\end{equation}
where $\vi{\mathbf{X}}{k}$ is the cluster of points in $\mathbf{X}$ assigned to the center $\vi{\Bc}{k}$.
\subsection{Location problems}
\label{subsec:Location}
Metric clustering problems often arise in location analysis, where $\mathbf{X}$ is the set of the locations of customers, $W$ is the set of their weights (or demands), and it is required to locate $K$ facilities $\{\vi{\Bc}{k}\}$ to serve the customers optimally in the sense of total weighted-distances traveled. The
problem (\textbf{L}.$K$) is then called a \textbf{multi--facility location problem}, or a \textbf{location--allocation problem} because it is required to locate the centers, and to assign or allocate the points to them.

Problem (\textbf{L}.$K$) is trivial for $K=N$ (every point is its own center) and reduces for $K=1$ to the \textbf{single facility location problem}: find the location of a \textbf{center} $\Bc\in\R^n$ so as to minimize the sum of weighted distances,
\begin{equation}
\min_{\dspl{\Bc\in\R^n}}\,\sum_{i=1}^N\,\vi{w}{i}\,d(\vi{\x}{i},\Bc).
\tag{\textbf{L}.$1$}
\end{equation}

For $1<K<N$, the problem (\textbf{L}.$K$) is NP-hard in general \cite{MEGIDDO-SUPOWIT}, while the planar case can be solved  polynomially in $N$, \cite{DREZNER-84}.
\subsection{Probabilistic approximation}
\label{subsec:Approximation}
(\textbf{L}.$K$) can be approximated by a continuous problem
\begin{equation}
\min_{\Bc_1,\cdots,\Bc_K}\,\sum_{k=1}^K\,\,\sum_{\vi{\x}{i}\in \mathbf{X}}\,\vi{w}{i}\,p_k(\vi{\x}{i})d(\vi{\x}{i},\vi{\Bc}{k}),
\tag{\textbf{P}.$K$}
\end{equation}
where rigid assignments $\vi{\x}{i}\in \vi{\X}{k}$ are replaced by probabilistic (soft) assignments, expressed by probabilities $p_k(\vi{\x}{i})$ that a point $\vi{\x}{i}$ belongs to the cluster $\vi{\X}{k}$.

For each point $\vi{\x}{i}$ the \textbf{cluster membership probabilities} $p_k(\vi{\x}{i})$ sum to 1, and are assumed to depend on the distances $d(\vi{\x}{i},\vi{\Bc}{k})$ as follows
\begin{equation*}
\fbox{membership in a cluster is more likely the closer is its center}
 \tag{\textbf{A}}
 \end{equation*}

Given these probabilities, the problem (P.$K$) can be decomposed into $K$ single facility location problems,
\begin{equation}
\min_{\Bc_k}\,\sum_{\vi{\x}{i}\in \mathbf{X}}\,p_k(\vi{\x}{i})\vi{w}{i}\,d(\vi{\x}{i},\vi{\Bc}{k}),
\tag{P.$k$}
\end{equation}
for $k\in\itoj{1}{K}$. The solutions $\vi{\Bc}{k}$ of the $K$ problems (P.$k$), are then used to calculate the new distances $d(\vi{\x}{i},\vi{\Bc}{k})$ for all $i\in\itoj{1}{N},\, k\in\itoj{1}{K}$, and from them, new probabilities $\{p_k(\vi{\x}{i})\}$, etc.
\subsection{The case for the $\ell_1$ norm}
\label{subsec:Dimension}
In high dimensions, distances between points become unreliable \cite{BEYER-ET-AL}, and this  in particular ``makes a proximity query meaningless and unstable because there is poor discrimination between the nearest and furthest neighbor" \cite{AGGARWAL-ET-AL-2000}. For the Euclidean distance
\begin{equation}
d_2(\x,\y)=
(\sum_{j=1}^n\,\big| \vj{\x}{j}-\vj{\y}{j} \big|^2)^{1/2}=
(\|\x\|_2^2-2\,\sum_{j=1}^n\,\vj{\x}{j}\,\vj{\y}{j}+\|\y\|_2^2)^{1/2}
\label{eq:L2-distance}
\end{equation}
between random points $\x,\y\in\R^n$, the cross products $\vj{\x}{j}\,\vj{\y}{j}$ in (\ref{eq:L2-distance}) tend to cancel for very large $n$, and consequently,
\[
d_2(\x,\y)\approx (\|\x\|_2^2+\|\y\|_2^2)^{1/2}.
\]
In particular, if $\x,\y$ are random points on the unit sphere in $\R^n$ then $d_2(\x,\y)\approx \sqrt{2}$ for very large $n$. This ``curse of high dimensionality'' limits the applicability of distance based methods in high dimension.

The $\ell_1$--distance is less sensitive to high dimensionality, and has been shown to ``provide the best discrimination in high--dimensional data spaces'',  \cite{AGGARWAL-ET-AL-2000}. We use it throughout this paper.
\subsection*{The plan of the paper}
The $\ell_1$--metric clustering problem is solved in \S~\ref{sec:L.1} for one center. A probabilistic approximation of (\textbf{L}.$K$) is discussed in \S~\ref{sec:Approximation}, the probabilities studied in \S\S~\ref{sec:Axioms}--\ref{sec:Probabilistic memberships}. The centers of the approximate problem are computed in \S~\ref{sec:Centers}.
Our main result, Algorithm PCM($\ell_1$) of \S~\ref{sec:Algorithm}, uses the power probabilities of \S~\ref{sec:Power probabilities}, and has running time that is linear in the dimension of the space, see Corollary~\ref{cor:Time}. Theorem 1, a monotonicity property of Algorithm~PCM($\ell_1$), is proved in \S~\ref{sec:Monotone}. Section \ref{sec:Conclusion} lists conclusions. Appendix A shows relations to previous work, and Appendix B reports some numerical results.
\section{The single facility location problem with the $\ell_1$--norm}
\label{sec:L.1}

For the $\ell_1$--distance (\ref{eq:L1-distance}) the problem (\textbf{L}.1) becomes
\begin{equation}
\min_{\Bc\in\R^n}\,\sum_{i=1}^N\,\vi{w}{i}\,d_1(\x_i,\Bc),\quad \text{or}\quad
\min_{\Bc\in\R^n}\,\sum_{i=1}^N\,\vi{w}{i}\,\sum_{j=1}^n\,
\big|\vij{\x}{i}{j}-\vj{\Bc}{j}\big|,
\label{eq:L1-ell1}
\end{equation}
in the variable $\Bc\in\R^n$, which can be solved separately for each component $\vj{\Bc}{j}$, giving the $n$ problems
\begin{equation}
\min_{\Bc[j]\in\R}\,\sum_{i=1}^N\,\vi{w}{i}\,
\big|\vij{\x}{i}{j}-\vj{\Bc}{j}\big|,\ j\in\itoj{1}{n}.
\label{eq:L1-j}
\end{equation}

\begin{definition}
\label{def:w-median}
Let $\mathbf{X}=\{\vi{x}{1},\cdots,\vi{x}{N}\}\in\R$ be an ordered set of points
\[
\vi{x}{1} \leq \vi{x}{2} \leq \cdots \leq \vi{x}{N}
 \]
and let $\mathbf{W}=\{\vi{w}{1},\cdots,\vi{w}{N}\}$ be a corresponding set of positive weights. A point $x$  is a \textbf{weighted median} (or $\mathbf{W}$--\textbf{median}) of $\mathbf{X}$ if there exist $\alpha,\beta\geq 0$ such that
\begin{equation}
\sum\,\{\vi{w}{i}:\,\vi{x}{i} < x\} + \alpha = \sum\,\{\vi{w}{i}:\,\vi{x}{i} > x\} + \beta
\label{eq:w-median}
\end{equation}
where $\alpha+\beta$ is the weight of $x$ if $x\in \mathbf{X}$, and $\alpha=\beta=0$ if $x\not\in \mathbf{X}$.
\end{definition}
The weighted median always exists, but is not necessarily unique.
\begin{lemma}
\label{lem:W-median}
For $\mathbf{X},\,\mathbf{W}$ as above,
define
\begin{equation}
\vi{\theta}{k}:=\frac{\sum\limits_{i=1}^k\,\vi{w}{i}}{\sum\limits_{i=1}^N\,\vi{w}{i}},\quad k\in\itoj{1}{N},
\label{eq:theta-k}
\end{equation}
and let $k^*$ be the smallest $k$ with $\vi{\theta}{k} \geq \tfrac{1}{2}$. If
\begin{equation}
\vi{\theta}{k^*}>\tfrac{1}{2}
\label{eq:theta>}
\end{equation}
then $\vi{x}{k^*}$ is the unique weighted median, with
\begin{equation}
\alpha = \tfrac{1}{2}\,\left(\vi{w}{k^*}+\sum_{k>k^*}\,\vi{w}{k}-\sum_{k<k^*}\,\vi{w}{k}\right),\ \beta = \vi{w}{k^*}-\alpha.
\label{eq:alpha,beta}
\end{equation}
Otherwise, if
\begin{equation}
\vi{\theta}{k^*}=\tfrac{1}{2},
\label{eq:theta<}
\end{equation}
then any point in the open interval $(\vi{x}{k^*},\vi{x}{k^*\!\!+\!\!1})$ is a weighted median with $\alpha=\beta=0$.
\end{lemma}
\begin{proof}
The statement holds since the sequence (\ref{eq:theta-k}) is increasing from $\theta_1=(w_1/\sum_{k=1}^N\,w_k)$ to $\theta_N=1$.
\end{proof}
\textbf{Note}: In case (\ref{eq:theta<}),
\[
\sum\,\{\vi{w}{k}:\,\vi{x}{k} \leq \vi{x}{k^*}\} = \sum\,\{\vi{w}{k}:\,\vi{x}{k} \geq  \vi{x}{k^*\!\!+\!\!1}\},
\]
we can take the median as the midpoint of $\vi{x}{k^*}$ and $\vi{x}{k^*\!\!+\!\!1}$, in order to conform with the classical definition of the median (for an even number of points of equal weight) .
\begin{lemma}
\label{lem:w-median}
Given $\mathbf{X}$ and $\mathbf{W}$ as in Definition~\ref{def:w-median}, the set of minimizers $c$ of
\[
\sum_{i=1}^N\,\vi{w}{i}\,\big|\vi{x}{i}-c\big|
\]
is the set of $\mathbf{W}$--medians of $\mathbf{X}$.
\end{lemma}
\begin{proof}
The result is well known if all weights are 1. If the weights are integers, consider a point $\vi{x}{i}$ with weight $\vi{w}{i}$ as $\vi{w}{i}$ coinciding points of weight 1 and the result follows. Same if the weights are rational. Finally, if the weights are real, consider their rational approximations.
\end{proof}
\section{Probabilistic approximation of (L.$K$)}
\label{sec:Approximation}

We relax the assignment problem in (\textbf{L}.$K$) of \S~\ref{subsec:Clustering problem} by using a \textbf{continuous approximation} as follows,
\begin{equation}
\min\,
\sum_{k=1}^K\,\sum_{i=1}^N\,\vi{w}{i}\,\vi{p}{k}(\vi{\x}{i})\,d(\vi{\x}{i},\vi{\Bc}{k})
\tag{\textbf{P}.$K$}
\end{equation}
with two sets of variables,

the \textbf{centers} $\{\vi{\Bc}{k}\}$, and

the \textbf{cluster membership probabilities} $\{\vi{p}{k}(\vi{\x}{i})\}$,
\begin{equation}
\vi{p}{k}(\vi{\x}{i}):=\text{Prob}\,\{\vi{\x}{i}\in\vi{\mathbf{X}}{k}\},\ i\in\itoj{1}{N},\,k\in\itoj{1}{K},
\label{eq:membership_probabilities}
\end{equation}

Because the probabilities
 $\{\vi{p}{k}(\vi{\x}{i})\}$ add to 1 for each $i\in\itoj{1}{N}$,  the objective function of (\textbf{P}.$K$) is an upper bound on the
optimal value of (\textbf{L}.$K$),
\begin{equation}
\sum_{k=1}^K\,\sum_{i=1}^N\,\vi{w}{i}\,\vi{p}{k}(\vi{\x}{i})\,d(\vi{\x}{i},\vi{\Bc}{k}) \geq \min\,(\text{\textbf{L}}.K),
\label{eq:f>P}
\end{equation}
and therefore so is the optimal value of (\textbf{P}.$K$),
\begin{equation}
\min\,(\text{\textbf{P}}.K) \geq \min\,(\text{\textbf{L}}.K).
\label{eq:P>L}
\end{equation}
\section{Axioms for probabilistic distance clustering}
\label{sec:Axioms}
In this section, $d_k(\x)$ stands for $d_k(\x,\Bc_k)$, the distance of $\x$ to the center $\Bc_k$ of the $\ith{k}{th}$--cluster,  $k\in\itoj{1}{K}$. To simplify notation, the point $\x$ is assumed to have weight $w=1$.

The \textbf{cluster membership probabilities} $\{p_k(\x):k\in\ovl{1,\!K}\}$ of a point $\x$ depend only on the \textbf{distances} $\{d_k(\x):\,k\in\ovl{1,\!K}\}$,
\begin{equation}
\mathbf{p}(\x)=\mathbf{f}(\mathbf{d}(\x))
\label{eq:p=f(d)}
\end{equation}
where $\mathbf{p}(\x)\in\R^K$ is the vector of probabilities $(p_k(\x))$, and $\mathbf{d}(\x)$ is the vector of distances $(d_k(\x))$. Natural assumptions for the relation (\ref{eq:p=f(d)}) include
\begin{subequations}
\begin{align}
d_i(\x)<d_j(\x) &\Longrightarrow\ p_i(\x)>p_j(\x),\ \text{for all}\ i,j\in\ovl{1,\!K}
\label{eq:properties-a}
\\
\mathbf{f}(\lambda\,\mathbf{d}(\x)) &=\mathbf{f}(\mathbf{d}(\x)),\ \text{for any}\ \lambda>0
\label{eq:properties-b}
\\
Q\,\mathbf{p}(\x)&=\mathbf{f}(Q\,\mathbf{d}(\x)),\ \text{for any permutation matrices}\ Q
\label{eq:properties-c}
\end{align}
\label{eq:properties}
\end{subequations}
Condition (\ref{eq:properties-a}) states that membership in a cluster is more probable the closer it is, which is Assumption (A) of \S~\ref{subsec:Approximation}.
The meaning of (\ref{eq:properties-b}) is that the probabilities $p_k(\x)$ do not depend on the scale of measurement, i.e., $\mathbf{f}$ is homogeneous of degree 0. It follows that the probabilities $p_k(\x)$ depend only on the ratios of the distances $\{d_k(\x):\,k\in\ovl{1,\!K}\}$.

The symmetry of $\mathbf{f}$, expressed by  (\ref{eq:properties-c}), guarantees for each $k\in\ovl{1,K}$, that the probability $p_k(\x)$ does not depend on the numbering of the other clusters.

Assuming continuity of $\mathbf{f}$ it follows from (\ref{eq:properties-a}) that
\[
d_i(\x)=d_j(\x)\ \Longrightarrow\ p_i(\x)=p_j(\x),
\]
for any $i,j\in\ovl{1,K}$. 

For any nonempty subset $\mathcal{S}\subset\ovl{1,K}$, let
\[
p_\mathcal{S}(\x)=\sum_{s\in\mathcal{S}}\,p_s(\x),
\]
the probability that $\x$ belongs to one of the clusters $\{\mathcal{C}_s:\,s\in\mathcal{S}\}$, and let $p_k(\x|\mathcal{S})$ denote the \textbf{conditional probability} that $\x$ belongs to the cluster $\mathcal{C}_k$, given that it belongs to one of the clusters $\{\mathcal{C}_s:\,s\in\mathcal{S}\}$.

Since the probabilities $p_k(\x)$ depend only on the ratios of the distances $\{d_k(\x):\,k\in\ovl{1,\!K}\}$, and these ratios are unchanged in subsets $\mathcal{S}$ of the index set $\ovl{1,\!K}$, it follows that for all $k\in\ovl{1,K},\ \emptyset\neq\mathcal{S}\subset\ovl{1,\!K}$,
\begin{equation}
p_k(\x) = p_k(\x|\mathcal{S})\,p_\mathcal{S}(\x)
\label{eq:LCA}
\end{equation}
which is the \textbf{choice axiom} of Luce, \cite[Axiom 1]{LUCE-59}, and therefore, \cite{YELLOT},
\begin{equation}
p_k(\x|\mathcal{S})=\frac{v_k(\x)}{\sum\limits_{s\in\mathcal{S}}\,v_s(\x)}
\label{eq:v(S)}
\end{equation}
where $v_k(\x)$ is a scale function, in particular,
\begin{equation}
p_k(\x)=\frac{v_k(\x)}{\sum\limits_{s\in\ovl{1,\!K}}\,v_s(\x)}\;.
\label{eq:v(x)}
\end{equation}
Assuming $v_k(\x)\neq 0$ for all $k$, it follows that
\begin{equation}
p_k(\x)v_k(\x)^{-1} = \frac{1}{\sum\limits_{s\in\ovl{1,\!K}}\,v_s(\x)},
\label{eq:pv}
\end{equation}
where the right hand side is a function of $\x$, and does not depend on $k$.

 Property (\ref{eq:properties-a}) implies that the function $v_k(\cdot)$ is a monotone decreasing function of $d_k(\x)$.

\section{Cluster membership probabilities as functions of distance}
\label{sec:Probabilistic memberships}
Given $K$ centers $\{\vi{\Bc}{k}\}$, and a point $\x$ with weight $w$ and distances $\{d(\x,\vi{\Bc}{k}):k\in\itoj{1}{K}\}$ from these centers, a simple choice for the function $v_k(\x)$ in (\ref{eq:v(S)}) is
\begin{equation}
  v_k(\x)=\frac{1}{w\,d_k(\x)}\;,
  \label{eq:v=1/d}
  \end{equation}
for which (\ref{eq:pv}) gives\footnote{There are other ways to model Assumption (\textbf{A}), e.g. \cite{BI-IY-08}, but the simple model (\ref{eq:pd=D}) works well enough for our purposes.},
\begin{equation}
w\,\vi{p}{k}(\x)\,d(\x,\vi{\Bc}{k})=D(\x),\ k\in\itoj{1}{K},
\label{eq:pd=D}
\end{equation}
where the function $D(\x)$, called the \textbf{joint distance function} (JDF) at $\x$, does not depend on $k$.

For a given point $\x$ and given centers $\{\Bc_k\}$, equations (\ref{eq:pd=D}) are optimality conditions for the extremum problem
\begin{equation}
\min\,\left\{w \sum_{k=1}^K\,p_k^2\,d(\x,\Bc_k):\,\sum_{k=1}^K\,p_k=1,\ p_k \geq 0,\
k\in\itoj{1}{K}\right\}
\label{eq:E}
\end{equation}
in the probabilities $\{\vi{p}{k} := \vi{p}{k}(\x)\}$. The squares of probabilities in the objective of
(\ref{eq:E}) serve to smooth the underlying non--smooth problem, see the seminal paper by Teboulle \cite{TEBOULLE-07}. Indeed, (\ref{eq:pd=D}) follows by differentiating the Lagrangian
\begin{equation}
L(\mathbf{p},\lambda)= w\,\sum_{k=1}^K\,\vi{p}{k}^2\,d(\x,\vi{\Bc}{k}) + \lambda\,\left(\sum_{k=1}^K\,\vi{p}{k}-1\right),
\label{eq:Lagrangian}
\end{equation}
with respect to $\vi{p}{k}$ and equating the derivative to zero.

Since probabilities add to one we get from (\ref{eq:pd=D}),
\begin{equation}
p_{\displaystyle k}(\x)=\frac{\prod\limits_{j\neq k}\,d(\x,\vi{\Bc}{j})}
{\sum\limits_{\ell=1}^K\,\prod\limits_{m\neq \ell}\,d(\x,\vi{\Bc}{m})},\ k\in\itoj{1}{K},
\label{eq:p(x)}
\end{equation}
and the JDF at $\x$,
\begin{equation}
D(\x)=w\,\frac{\prod\limits_{j=1}^K\,d(\x,\vi{\Bc}{j})}
{\sum\limits_{\ell=1}^K\,\prod\limits_{m\neq \ell}\,d(\x,\vi{\Bc}{m})}\;,
\label{eq:D(x)}
\end{equation}
which is (up to a constant) the \textbf{harmonic mean} of the distances $\{d(\x,\vi{\Bc}{k}):k\in\itoj{1}{K}\}$, see also (\ref{eq:K-D(x)}) below.

Note that the probabilities $\{\vi{p}{k}(\x):\,k\in\itoj{1}{K}\}$ are determined by the centers $\{\vi{\Bc}{k}:\,k\in\itoj{1}{K}\}$ alone, while the function $D(\x)$ depends also on the weight $w$.
For example, in case $K=2$,
\begin{subequations}
\begin{gather}
\vi{p}{1}(\x)=\frac{d(\x,\vi{\Bc}{2})}{d(\x,\vi{\Bc}{1})+d(\x,\vi{\Bc}{2})},\
\vi{p}{2}(\x)=\frac{d(\x,\vi{\Bc}{1})}{d(\x,\vi{\Bc}{1})+d(\x,\vi{\Bc}{2})},
\label{eq:p(x)-2}
\\
D(\x)=w\,\frac{d(\x,\vi{\Bc}{1})\,d(\x,\vi{\Bc}{2})}{d(\x,\vi{\Bc}{1})+d(\x,\vi{\Bc}{2})}.
\label{eq:D(x)-2}
\end{gather}
\end{subequations}

\section{Computation of centers}
\label{sec:Centers}

We use the $\ell_1$--distance (\ref{eq:L1-distance}) throughout.
The objective function of (\textbf{P}.$K$)  is a separable function of the cluster centers,
\begin{subequations}
\begin{align}
f(\vi{\Bc}{1},\ldots,\vi{\Bc}{K}) &
:=\sum\limits_{k=1}^K\,\vi{f}{k}(\vi{\Bc}{k}),\label{eq:f(centers)}\\
\text{where}\quad
\vi{f}{k}(\Bc) &:=\sum\limits_{i=1}^N\,\vi{w}{i}\,\vi{p}{k}(\vi{\x}{i})\,\vi{d}{1}(\vi{\x}{i},\Bc),\quad k\in\itoj{1}{K}.\label{eq:f(c_k)}
\end{align}
\label{eq:f}
\end{subequations}
The centers problem thus separates into $K$ problems of type (\ref{eq:L1-ell1}),
\begin{equation}
\min_{\displaystyle{\Bc_k}\in\R^n}\,\sum_{i=1}^N\,\vi{w}{i}\,\vi{p}{k}(\vi{\x}{i})\,\sum_{j=1}^n\,
\big|\vij{\x}{i}{j}-\vij{\Bc}{k}{j}\big|,\ k\in\itoj{1}{K},
\label{eq:L1-k}
\end{equation}
coupled by the probabilities $\{\vi{p}{k}(\vi{\x}{i})\}$.
Each of these problems separates into $n$ problems of type  (\ref{eq:L1-j}) for the components $\vij{\Bc}{k}{j}$,
\begin{equation}
\min_{\displaystyle{c_k[j]\in\R}}\,\sum_{i=1}^N\,\vi{w}{i}\,\vi{p}{k}(\vi{\x}{i})\,
\big|\vij{\x}{i}{j}-\vij{\Bc}{k}{j}\big|,\ k\in\itoj{1}{K},\ j\in\itoj{1}{n},
\label{eq:LP-w-j}
\end{equation}
whose solution, by Lemma~\ref{lem:w-median}, is a weighted median of the points $\{\vij{\x}{i}{j}\}$ with weights $\{\vi{w}{i}\,\vi{p}{k}(\vi{\x}{i})\}$.
%
\section{Power probabilities}
\label{sec:Power probabilities}
The cluster membership probabilities $\{p_k(\x):\,k\in\itoj{1}{K}\}$ of a point $\x$ serve to relax the rigid assignment of $\x$ to any of the clusters, but eventually it may be necessary to produce such an assignment. One way to achieve this is  to raise the membership probabilities $p_k(\x)$ of (\ref{eq:p(x)}) to a power $\nu \geq 1$, and normalize, obtaining the \textbf{power probabilities}
\begin{align}
\vi{p}{k}^{(\nu)}(\x) &:=\frac{\vi{p}{k}^{\nu}(\x)}{\sum\limits_{j=1}^K\,\vi{p}{j}^{\nu}(\x)},
\label{eq:p(x)-nu-normalized}
\\
\intertext{which, by (\ref{eq:p(x)}), can also be expressed in terms of the distances $d(\x,\vi{\Bc}{k})$,}
\vi{p}{k}^{(\nu)}(\x) &:=\frac{\prod\limits_{j\neq k}\,d(\x,\vi{\Bc}{j})^\nu}
{\sum\limits_{\ell=1}^K\,\prod\limits_{m\neq \ell}\,d(\x,\vi{\Bc}{m})^\nu},\ k\in\itoj{1}{K}.
\label{eq:p(x)-nu}
\end{align}

As the exponent $\nu$ increases the power probabilities $\vi{p}{k}^{(\nu)}(\x)$ tend to hard assignments: If $M$ is the index set of maximal probabilities, and $M$ has $\#M$ elements, then,
\begin{equation}
\lim_{\nu\to\infty}\,\vi{p}{k}^{(\nu)}(\x)=\left\{
                                           \begin{array}{cl}
                                             \frac{1}{\#M}, & \hbox{if $k\in M$;} \\
                                             0, & \hbox{otherwise,}
                                           \end{array}
                                         \right.
\label{eq:limit}
\end{equation}
and the limit is a hard assignment if $\#M=1$, i.e. if the maximal probability is unique.

Numerical experience suggests an increase of $\nu$ at each iteration, see, e.g., (\ref{eq:nu+}) below.
\section{Algorithm PCM($\ell_1$): Probabilistic Clustering Method with $\ell_1$ distances}
\label{sec:Algorithm}
The problem (\textbf{P}.$K$) 
is solved iteratively, using the following updates in succession.


\textbf{Probabilities computation}: Given $K$ centers $\{\vi{\Bc}{k}\}$, the assignments probabilities $\{p_k^{(\nu)}(\vi{\x}{i})\}$ are calculated using (\ref{eq:p(x)-nu}). The exponent $\nu$ is updated at each iteration, say by a constant increment $\Delta \geq 0$,
\begin{equation}
\nu:=\nu+\Delta
\label{eq:nu+}
\end{equation}
starting with an initial $\nu_0$.

\textbf{Centers computation}: Given the assignment probabilities $\{p_k^{(\nu)}(\vi{\x}{i})\}$, the problem (\textbf{P}.$K$) separates into $K n$ problems of type (\ref{eq:LP-w-j}),
\begin{equation}
\min_{\Bc_k[j]\in\R}\,\sum_{i=1}^N\,\vi{w}{i}\,p^{(\nu)}_{\displaystyle k}(\vi{\x}{i})\,
\big|\vij{\x}{i}{j}-\vij{\Bc}{k}{j}\big|,\ k\in\itoj{1}{K},\ j\in\itoj{1}{n},
\label{eq:LP-w-j.nu}
\end{equation}
one for each component $\vij{\Bc}{k}{j}$ of each center $\vi{\Bc}{k}$, that are solved by Lemma~\ref{lem:w-median}.


These results are presented in an algorithm form as follows.
\medskip\\
\textbf{Algorithm PCM($\ell_1$)}
: An algorithm for the $\ell_1$ clustering problem\\
\rule{120mm}{0.2mm}
\begin{tabbing}
xxxxxxx \= xxxxxxxxxxxxxxxxxxxxxxxxxxxxxxxxxxxx \kill
\textbf{Data}: \>  $\mathbf{X}=\{\vi{\x}{i}:\,i\in\itoj{1}{N}\}$ data points,\  $\{\vi{w}{i}:\,i\in\itoj{1}{N}\}$ their weights,\\
\> $K$ the number of clusters,\\
\> $\epsilon > 0$ (stopping criterion),\\
\> $\nu_0 \geq 1$ (initial value of the exponent $\nu$),\ $\Delta > 0$ (the increment in (\ref{eq:nu+}).)\\[2mm]
\pushtabs
\textbf{Initialization}: $K$ arbitrary centers $\{\vi{\Bc}{k}:\,k\in\itoj{1}{K}\}$, $\nu:=\nu_0$.\\[2mm]
\textbf{Iteration}:\\[2mm]
\poptabs
Step 1 \>  \textbf{compute} distances $\{\vi{d}{1}(\x,\vi{\Bc}{k}):\,k\in\itoj{1}{K}\}$ for
all $\x\in\mathbf{X}$\\[2mm]
Step 2 \>  \textbf{compute} the assignments $\{p^{(\nu)}_{\displaystyle k}(\x):\,\x\in\mathbf{X},\ k\in\itoj{1}{K}\}$
(using (\ref{eq:p(x)-nu}))\\[2mm]
Step 3 \> \textbf{compute} the new centers $\{\vi{\Bc}{k+}:\,k\in\itoj{1}{K}\}$
(applying Lemma~\ref{lem:w-median} to (\ref{eq:LP-w-j.nu}))\\[2mm]
Step 4 \> \textbf{if}
$\sum\limits_{k=1}^K\,\vi{d}{1}(\vi{\Bc}{k+},\vi{\Bc}{k})< \epsilon$\quad \textbf{stop}\\[2mm]
\> \textbf{else} $\nu:=\nu+\Delta$ , \textbf{return} to step 1
 \end{tabbing}
 \rule{120mm}{0.2mm}
\begin{corollary}
\label{cor:Time}
The running time of Algorithm PCM($\ell_1$) is
\begin{equation}
O(NK(K^2 + n)I),
\label{eq:Time}
\end{equation} where $n$ is the dimension of the space, $N$ the number of points, $K$ the number of clusters, and $I$ is the number of iterations.
\end{corollary}
\begin{proof}
The number of operations in an iteration is calculated as follows:

Step 1: $O(nNK)$, since computing the $\ell_1$ distance between two $n$-dimensional vectors takes $O(n)$ time, and there are $N\,K$ distances between all points and all centers.

Step 2: $O(NK^3)$, there are $NK$ assignments, each taking $O(K^2)$.

Step 3: $O(nNK)$, computing the weighted median of $N$ points in $\R$ takes $O(N)$ time, and $K\,n$ such medians are computed.

Step 4: $O(nK)$, since there are $K$ cluster centers of dimension $n$.

The corollary is proved by combining the above results.
\end{proof}
\begin{remark}
\label{rem:Cor}\ \\
(a) The result (\ref{eq:Time}) shows that Algorithm~PCM($\ell_1$) is linear in $n$, which in high--dimensional data is much greater than $N$ and $K$.\\
(b) The first few iterations of the algorithm come close to the final centers, and thereafter the iterations are slow, making the stopping rule in Step 4 ineffective. A better stopping rule is a bound on the number of iterations $I$, which can then be taken as a constant in (\ref{eq:Time}).\\
(c) Algorithm~PCM($\ell_1$) can be modified to account for very unequal cluster sizes, as in \cite{IY-BI-08}. This modification did not significantly improve the performance of the algorithm in our experiments.\\
(d)
The centers here are computed from scratch at each iteration using the current probabilities, unlike the Weiszfeld method  \cite{Weiszfeld-37} or its generalizations, \cite{IY-BI-10}--\cite{IY-BI-11},  where the centers are updated at each iteration.
\end{remark}
\section{Monotonicity}
\label{sec:Monotone}
The centers computed iteratively by Algorithm~PCM($\ell_1$) 
 are confined to the convex hull of $\mathbf{X}$, a compact set, and therefore a subsequence converges to an optimal solution of the approximate problem (\textbf{P}.$K$), that in general is not an optimal solution of the original problem (\textbf{L}.$K$).

The JDF of the data set $\X$ is defined as the sum of the JDF's of its points,
\begin{equation}
D(\mathbf{X}):=\sum\limits_{\mathbf{x}\in \mathbf{X}}\,D(\x).
\label{eq:D(X)}
\end{equation}
We prove next a monotonicity result for $D(\X)$.
\begin{theorem}
\label{th:Monotone}
The function $D(\mathbf{X})$
decrease along any sequence of iterates of centers.
\end{theorem}
\begin{proof}
The function $D(\X)$ can be written as
\begin{align}
D(\mathbf{X}) &:= \sum_{\mathbf{x}\in \mathbf{X}}\,\left(\sum_{k=1}^K\,p_k(\x)\right)\,D(\x),\ \text{since the probabilities add to 1},
\notag\\
&=\sum_{\mathbf{x}\in \mathbf{X}}\,\sum_{k=1}^K\,w(\mathbf{x})\,p_k(\mathbf{x})^2\,d_1(\x,\vi{\Bc}{k}),\ \text{by (\ref{eq:pd=D})}.
\label{eq:sumx-w-p^2-d}
\end{align}
The proof is completed by noting that, for each $\x$, the probabilities $\{p_k(\x):k\in\itoj{1}{K}\}$ are chosen as to minimize the function
\begin{equation}
\sum_{k=1}^K\,w(\x)\,p_k(\x)^2\,d_1(\mathbf{x},\vi{\Bc}{k})
\label{eq:sum-w-p^2-d}
\end{equation}
for the given centers, see (\ref{eq:E}), and the centers $\{\vi{\Bc}{k}:k\in\itoj{1}{K}\}$ minimize
the function (\ref{eq:sum-w-p^2-d}) for the given probabilities.
\end{proof}
\begin{remark}
\label{rem:nu}The function $D(\mathbf{X})$ also decreases if the exponent $\nu$ is increased in (\ref{eq:p(x)-nu-normalized}), for then shorter distances are becoming more probable in (\ref{eq:sumx-w-p^2-d}).
\end{remark}
\section{Conclusions}
\label{sec:Conclusion}
In summary, our approach has the following advantages.
\begin{enumerate}
\item In numerical experiments, see Appendix B, Algorithm PCM($\ell_1$) outperformed the fuzzy clustering $\ell_1$--method, the $K$--means $\ell_1$ method, and the generalized Weiszfeld method \cite{IY-BI-10}.
\item The solutions  of (\ref{eq:E}) are less sensitive to outliers than the solutions of (\ref{eq:FCM}), which uses squares of distances.
\item The probabilistic principle (\ref{eq:pd=f}) allows using other monotonic functions, in particular the exponential function $\phi(d)=e^d$, that gives sharper results, and requires only that every distance $d(\x,\Bc)$ be replaced by $\exp\,\{d(\x,\Bc)\}$, \cite{BI-IY-08}.
\item The JDF (\ref{eq:D(X)}) of the data set, provides a guide to the ``right'' number of clusters for the given data, \cite{BI-IY-10}.
\end{enumerate}

\end{spacing}

\renewcommand{\theequation}{A-\arabic{equation}}
  \setcounter{equation}{0}  
\appendix
\begin{center}\textbf{Appendix A: Relation to previous work}\end{center}
Our work brings together ideas from four different areas: inverse distance weighted interpolation, fuzzy clustering, subjective probability, and optimality principles.

\textbf{1. Inverse distance weighted} (or \textbf{IDW}) \textbf{interpolation} was introduced in 1965 by Donald Shepard, who published his results \cite{SHEPARD-68} in 1968. Shepard, then an undergraduate at Harvard, worked on the following problem:

\emph{A function $u:\R^n\to\R$ is evaluated at $K$ given points $\{\x_k:k\in\itoj{1}{K}\}$ in $\R^n$, giving the values $\{u_k:k\in\itoj{1}{K}\}$, respectively. These values are the only information about the function. It is required to estimate $u$ at any point $\x$.}

Examples of such functions include rainfall in meteorology, and  altitude in topography. The point $\x$ cannot be too far from the data points, and ideally lies in their convex hull.

Shepard estimated the value $u(\x)$ as a convex combination of the given values $u_k$,
\begin{align}
u(\x)&=\sum_{k=1}^K\,\lambda_k(\x)\,u_k
\label{eq:Shep-1}
\\
\intertext{where the weights $\lambda_k(\x)$ are inversely ptoportional to the distances $d(\x,\x_k)$ between $\x$ and $\x_k$, say}
u(\x) &= \sum_{k=1}^K\,\left(\dfrac{\dfrac{1}{d(\x,\x_k)}}{\sum\limits_{j=1}^K\,\dfrac{1}{d(\x,\x_j)}}\right)\,u_k
\label{eq:Shep-2}
\\
\intertext{giving the weights}
\lambda_k(\x)&=\dfrac{\prod\limits_{j\neq k}\,d(\x,\x_j)}{\sum\limits_{\ell=1}^K\,\prod\limits_{m\neq \ell}\,d(\x,\x_m)}
\label{eq:Shep-3}
\end{align}
that are identical with the probabilities (\ref{eq:p(x)}), if the data points are identified with the centers. IDW interpolation is used widely in spatial data analysis, geology, geography, ecology and related areas.

Interpolating the $K$ distances $d(\x,\x_k)$, i.e. taking $u_k=d(\x,\x_k)$ in (\ref{eq:Shep-2}), gives
\begin{equation}
K\,\dfrac{\prod\limits_{j=1}^K\,d(\x,\x_j)}{\sum\limits_{\ell=1}^K\,\prod\limits_{m\neq \ell}\,d(\x,\x_m)}
\label{eq:K-D(x)}
\end{equation}
the harmonic mean of the distances $\{d(\x,\x_k):\,k\in\itoj{1}{K}\}$, which is the JDF in (\ref{eq:D(x)}) multiplied by a scalar.

The harmonic mean pops up in several areas of spatial data analysis. In 1980 Dixon and Chapman \cite{DIXON-CHAPMAN-80} posited that the \emph{home--range} of a species is a contour of the harmonic mean of the areas it frequents, and this has since been confirmed for  hundreds of species. The importance of the harmonic mean in clustering was established by Teboulle \cite{TEBOULLE-07}, Stanforth, Kolossov and Mirkin \cite{SANFORTH-07}, Zhang, Hsu, and Dayal \cite{ZHANG}--\cite{ZHANG-PATENT}, Ben--Israel and Iyigun \cite{BI-IY-08} and others. Arav \cite{ARAV-07} showed the harmonic mean of distances to satisfy a system of reasonable axioms for contour approximation of data.

\textbf{2. Fuzzy clustering} introduced by J.C. Bezdek in 1973, \cite{BEZDEK-73}, is a relaxation of the original problem, replacing
the hard assignments of points to clusters by soft, or fuzzy, assignments of points simultaneously to all clusters, the strength of association of $\vi{\x}{i}$ with the $\ith{k}{th}$ cluster is measured by $w_{ik}\in [0,1]$.

 In the fuzzy c--means (FCM) method \cite{BEZDEK-81} the centers $\{\vi{\Bc}{k}\}$ are computed by
 \begin{equation}
 \min\,\sum_{i=1}^N\,\sum_{k=1}^K\,w_{ik}^m\,\|\vi{\x}{i}-\vi{\Bc}{k}\|_2^2,
 \label{eq:FCM}
 \end{equation}
 where the weights $x_{ik}$ are updated as\footnote{The weights (\ref{eq:wik}) are optimal for the problem (\ref{eq:FCM}) if they are probabilities, i.e. if they are required to add to 1 for every point $\x_i$.}
 \begin{equation}
 w_{ik}= \frac{1}{\sum\limits_{j=1}^K\,\left(\dfrac{\|\vi{\x}{i}-\vi{\Bc}{k}\|_2}
 {\|\vi{\x}{i}-\vi{\Bc}{j}\|_2}\right)^{2/m-1}},
 \label{eq:wik}
 \end{equation}
 and the centers are then calculated as convex combinations of the data points,
 \begin{equation}
 \Bc_k = \sum_{i=1}^N\,\left(\dfrac{w_{ik}^m}{\sum\limits_{j=1}^N\,w_{jk}^m}\right)\,\x_i.\ k\in\itoj{1}{K}.
 \label{eq:c-fuzz}
 \end{equation}
The constant $m\geq 1$ (the ``fuzzifier'') controls he fuzziness of the assignments, which become hard assignments in the limit as $m\downarrow 1$. For $m=1$, FCM is the classical $K$--means method. If $m=2$ then the weights $w_{ik}$ are inversely proportional to the square distance $\|\vi{\x}{i}-\vi{\Bc}{k}\|_2^2$, analogously to (\ref{eq:pd=D}).

Fuzzy clustering is one of the best known, and most widely used, clustering methods. However, it may need some modification if the data in question is very high--dimensional, see, e.g. \cite{KLAWONN-13}.


\textbf{3. Subjective probability}. There is some arbitrariness in the choice of the model and the fuzzifier $m$ in (\ref{eq:FCM})--(\ref{eq:wik}). In contrast, the probabilities (\ref{eq:p(x)}) can be justified axiomatically. Using ideas and classical results (\cite{LUCE-59}, \cite{YELLOT}) from subjective probability it is shown in Appendix B that the cluster membership probabilities $p_k(\x)$, and distances $d_k(\x)$, satisfy an inverse relationship, such as,
\begin{equation}
p_k(\x)\,\phi(d(\x,\Bc_k)) = f(\x),\ k\in\itoj{1}{K},
\label{eq:pd=f}
\end{equation}
where $\phi(\cdot)$ is non--decreasing, and $f(\x)$ does not depend on $k$. In particular, the choice $\phi (d)=d$ gives (\ref{eq:pd=D}), which works well in practice.

\textbf{4. Optimality principle}. Equation  (\ref{eq:pd=f}) is a necessary optimality condition for the problem
\begin{equation}
\min\,\left\{\sum_{k=1}^K\,p^2\,\phi(d(\x,\Bc_k)):\,\sum_{k=1}^K\,p_k =1,\ p_k\geq 0, k\in\itoj{1}{K}\right\},
\label{eq:Opt}
\end{equation}
that reduces to (\ref{eq:E}) for the choice $\phi(d)=d$. This shows the probabilities $\{p_k(\x)\}$ of (\ref{eq:p(x)}) to be optimal, for the model chosen.

\begin{remark}
\label{rem:Kelvin}
Minimizing a function of squares of probabilities seems unnatural, so a physical analogy may help. Consider an electric circuit with $K$ resistances $\{R_k\}$ connected in parallel. A current $I$ through the circuit splits into $K$ currents, with current $I_k$ through the resistance $R_k$. These currents solve an optimization problem (the \textbf{Kelvin principle})
\begin{equation}
\min_{I_1,\cdots,I_K}\,\left\{\sum_{k=1}^K\,I_k^2\,R_k:\,\sum_{k=1}^K\,I_k = I\right\}
\label{eq:Kelvin}
\end{equation}
that is analogous to (\ref{eq:E}). The optimality condition for (\ref{eq:Kelvin}) is \textbf{Ohm's law},
\[
I_k\,R_k = \text{constant}
\]
a statement that potential is well defined, and an analog of (\ref{eq:pd=D}). The equivalent resistance of the circuit, i.e. the resistance $R$ such that $I^2\,R$ is equal to the minimal value in (\ref{eq:Kelvin}), is then the JDF (\ref{eq:D(x)}) with $R_j$ instead of $d(\x,\vi{\Bc}{j})$ and $w=1$.
\end{remark}
\renewcommand{\theequation}{B-\arabic{equation}}
  \setcounter{equation}{0}  
\appendix
\begin{center}\textbf{Appendix B: Numerical Examples}\end{center}

In the following examples we use synthetic data to be clustered into $K=2$ clusters. The data consists of two randomly generated clusters, $\mathbf{X}_1$ with $N_1$ points, and $\mathbf{X}_2$ with $N_2$ points.

The data points $\x=(x_1,\cdots,x_n)\in\R^n$ of cluster $\mathbf{X}_k$ are such that all of their components $x_i, 1\leq i\leq n$ are generated by sampling from a distribution $F_k$ with mean $\mu_k$, $k=1,2$.
In cluster $\mathbf{X}_1$ we take $\mu_1=1$, and in cluster $\mathbf{X}_2$, $\mu_2=-1$.

We ran Algorithm~PCM($\ell_1$), with the parameters $\nu_0=1,\,\Delta=0.1$, and compared its performance with that of the fuzzy clustering method \cite{BEZDEK-81} with the $\ell_1$ norm, as well as the generalized Weiszfeld    algorithm of \cite{IY-BI-11} (that uses Euclidean distances), and the $\ell_1$--K-Means algorithm \cite{matlab}. For each method we used a stopping rule of at most 100 iterations (for Algorithm~PCM($\ell_1$) this replaces Step 4).
For each experiment we record the average percentage of misclassification (a misclassification occurs when a point in $\mathbf{X}_1$ is declared to be in $\mathbf{X}_2$, or vice versa) from 10 independent problems.
In examples \ref{ex:1},\ref{ex:2},\ref{ex:3} we choose the probability distributions to be $F_k = N(\mu_k, \sigma)$.
\begin{example}
\label{ex:1}
In this example the clusters are of equal size, $N_1=N_2=100$. Table~\ref{table:1} gives the percentages of misclassification under the five methods tested, for different values of $\sigma$ and dimension $n$.
\begin{table}
\label{table:1}
\begin{tabular}{|l|l|r|r|r|r|r|}
\hline
$\sigma$                          &        Method         &     $n=10^4$               &                      $n=5\cdot  10^4$  &     $n=10^5$  &     $n=5\cdot  10^5$  &    $n=10^6$  \\
\hline
$\sigma                           =        8$             &   PCM  ($\ell_1$)     &                      0.0        &      0.0   &         0.0   &          0.0    &    0.0\\     &      FCM  ($\ell_1$)  &           27.1  &    38.6  &    24.4  &    40.9  &    40.1\\   &      $K$-means  ($\ell_1$)  &           28.9  &     26.8  &     12.7  &     22.4  &     22.9\\  &       Gen.  Weiszfeld  &          48.5  &     48.8  &     48.0  &     48.2  &     47.9\\\hline
$\sigma                           =        16$            &   PCM  ($\ell_1$)     &                      4.3        &      0.0   &         0.0   &          4.7    &    0.0\\     &      FCM  ($\ell_1$)  &           41.0  &    42.1  &    44.5  &    43.9  &    39.5\\   &          $K$-means   ($\ell_1$)  &     41.8  &     35.2  &     23.7  &     23.5  &       23.6\\  &     Gen.       Weiszfeld  &     48.0  &     47.0  &     48.4  &     48.6  &             48.0\\\hline
$\sigma                           =        24$            &   PCM  ($\ell_1$)     &                      42.6       &      8.8   &         0.8   &          4.8    &    0.0\\     &      FCM  ($\ell_1$)  &           46.4  &    45.9  &    47.5  &    39.5  &    45.1\\   &      $K$-means  ($\ell_1$)  &           45.5  &     42.6  &     35.6  &     28.0  &     24.5\\  &       Gen.  Weiszfeld  &          47.9  &     47.8  &     47.1  &     48.0  &     48.2\\\hline
$\sigma                           =        32$            &   PCM  ($\ell_1$)     &                      46.0       &      42.2  &         13.4  &          13.6   &    0.0\\     &      FCM  ($\ell_1$)  &           47.4  &    46.0  &    44.8  &    46.0  &    46.0\\   &      $K$-means  ($\ell_1$)  &           46.4  &     45.7  &     40.3  &     36.0  &     30.7\\  &       Gen.  Weiszfeld  &          48.2  &     48.9  &     48.5  &     48.9  &     47.8\\\hline
\end{tabular}
\smallskip
\caption{Percentages              of       misclassified  data  in                     Example~\ref{ex:1}}
\end{table}

\end{example}

\begin{example}
\label{ex:2}
We use $N_1 = 200$ and $N_2 = 100$. Table~\ref{table:2} gives the percentages of misclassifications for different values of $\sigma$ and dimension $n$.
\begin{table}
\label{table:2}
\begin{tabular}{|l|l|r|r|r|r|r|}
\hline
$\sigma$                          &        Method         &     $n=10^4$               &                      $n=5\cdot  10^4$  &     $n=10^5$  &     $n=5\cdot  10^5$  &    $n=10^6$  \\
\hline
$\sigma                           =        8$             &   PCM  ($\ell_1$)    &                      0.0        &      0.0   &         0.0   &          0.0    &    0.0\\     &      FCM  ($\ell_1$)  &           11.9  &    19.1  &    25.2  &    30.1  &    22.6\\ &       $K$-means  ($\ell_1$)  &           20.8        &     25.9  &     18.4  &     31.4  &     13.6\\  &       Gen.   Weiszfeld  &          37.8       &     37.9  &     37.2  &     36.7  &     36.4\\\hline
$\sigma                           =        16$            &   PCM  ($\ell_1$)    &                      10.4       &      0.0   &         0.0   &          0.0    &    0.0\\     &      FCM  ($\ell_1$)  &           37.7  &    35.6  &    35.0  &    36.2  &    39.4\\ &       $K$-means  ($\ell_1$)  &           35.8  &     32.0  &     23.6  &     31.7  &       14.1\\  &      Gen.       Weiszfeld  &          38.0  &     37.7  &     35.8  &     36.6  &             37.8\\\hline
$\sigma                           =        24$            &   PCM  ($\ell_1$)    &                      44.1       &      5.9   &         1.2   &          0.0    &    0.0\\     &      FCM  ($\ell_1$)  &           41.3  &    37.7  &    38.9  &    36.7  &    34.6\\ &       $K$-means  ($\ell_1$)  &           40.3        &     39.9  &     32.7  &     33.3  &     15.5\\  &       Gen.   Weiszfeld  &          36.8       &     37.7  &     36.7  &     36.9  &     37.2\\\hline
$\sigma                           =        32$            &   PCM  ($\ell_1$)    &                      47.2       &      38.7  &         18.5  &          0.0    &    0.0\\     &      FCM  ($\ell_1$)  &           42.3  &    38.8  &    37.0  &    39.7  &    38.9\\ &       $K$-means  ($\ell_1$)  &           41.5        &     42.9  &     37.2  &     36.8  &     22.6\\  &       Gen.   Weiszfeld  &          36.7       &     36.9  &     36.0  &     36.5  &     37.4\\\hline
\end{tabular}
\smallskip
\caption{Percentages              of       misclassified  data  in                     Example~\ref{ex:2}}
\end{table}

\end{example}

\begin{example}
\label{ex:3}
In this case $N_1=1000,\,N_2=10$. The percentages of misclassification are included in Table~\ref{table:3}.
\begin{table}
\label{table:3}
\begin{tabular}{|l|l|r|r|r|r|r|}
\hline
$\sigma$                          &   Method         &     $n=10^3$               &                    $n=5\cdot  10^3$  &     $n=10^4$  &     $n=5\cdot  10^4$  &  $n=10^5$  \\
\hline                                                                     
$\sigma                           =   0.4$           &    PCM  ($\ell_1$)   &                    46.4       &      41.1  &         24.1  &          5.1    &  0.9\\     &   FCM  ($\ell_1$)  &  13.4  &  0.5   &  0.0   &  19.5  &  32.0\\ &  $K$-means  ($\ell_1$)  &  37.4  &  31.6  &  27.1  &     36.5  &     32.6\\  &       Gen.  Weiszfeld  &          35.4  &     36.7  &     32.6  &     33.7  &     38.7\\
\hline                                                                     
$\sigma                           =   0.8$           &    PCM  ($\ell_1$)   &                    47.4       &      31.4  &         23.4  &          5.4    &  1.8\\     &   FCM  ($\ell_1$)  &  29.3  &  9.5   &  13.0  &  27.3  &  37.2\\ &  $K$-means  ($\ell_1$)  &  37.5  &  32.0  &  27.1  &     36.4  &     32.6\\  &       Gen.  Weiszfeld  &          30.3  &     31.6  &     25.9  &     27.9  &     34.5\\
\hline                                                                     
$\sigma                           =   1.2$           &    PCM  ($\ell_1$)   &                    47.3       &      33.9  &         26.2  &          7.7    &  1.6\\     &   FCM  ($\ell_1$)  &  36.4  &  20.8  &  23.2  &  31.1  &  23.9\\ &  $K$-means  ($\ell_1$)  &  38.4  &  32.   2  &     28.3  &     36.4  &       32.6\\  &     Gen.       Weiszfeld  &     22.1  &     23.8  &     26.8  &     21.6  &       25.5\\
\hline                                                                     
$\sigma                           =   1.6$           &    PCM  ($\ell_1$)   &                    47.8       &      35.4  &         27.9  &          9.8    &  3.6\\     &   FCM  ($\ell_1$)  &  41.1  &  27.8  &  30.0  &  27.6  &  24.2\\ &  $K$-means  ($\ell_1$)  &  37.6  &  32.3  &  28.3  &     36.4  &     33.4\\  &       Gen.  Weiszfeld  &          23.1  &     23.2  &     21.1  &     25.4  &     31.6\\
\hline                                                                     
\end{tabular}
\smallskip
\caption{Percentages              of  misclassified  data  in                     Example~\ref{ex:3}}
\end{table}

\end{example}

In addition to experiments with normal data, we also consider instances with uniform data in Examples \ref{ex:4} and \ref{ex:5}.
In this case $F_k$ is a uniform distribution with mean $\mu_k$ and support length $|\mbox{supp}(F_k)|$.
\begin{example}
\label{ex:4}
We use $N_1=100,\,N_2=100$. The results are shown in Table~\ref{table:4}.
\begin{table}
\label{table:4}
\begin{tabular}{|l|l|r|r|r|r|r|}
\hline
$|\mbox{supp}(F)|$                &    Method         &                      $n=10^4$  &                    $n=5\cdot  10^4$  &     $n=10^5$  &       $n=5\cdot  10^5$  &          $n=10^6$  \\
\hline
$|\mbox{supp}(F)|=                8$   &             PCM   ($\ell_1$)    &         0.0                  &          0.0    &     0.0       &       0.0        &      0.0\\      &         FCM   ($\ell_1$)  &     0.0   &     0.1   &     0.3   &             2.7   &  0.1\\   &  $K$-means  ($\ell_1$)  &  5.0   &  5.0   &  4.8   &  0.0   &  0.0\\   &  Gen.  Weiszfeld  &  0.0   &  0.0   &  0.0   &  0.0   &  0.0\\\hline
$|\mbox{supp}(F)|=                16$  &             PCM   ($\ell_1$)    &         0.0                  &          0.0    &     0.0       &       0.0        &      0.0\\      &         FCM   ($\ell_1$)  &     8.9   &     29.1  &     26.6  &             21.9  &  25.6\\  &  $K$-means  ($\ell_1$)  &  23.8  &  25.9  &  18.0  &  23.4  &  17.8\\  &  Gen.  Weiszfeld  &  47.0  &  49.2  &  46.8  &  46.2  &  47.4\\\hline
$|\mbox{supp}(F)|=                24$  &             PCM   ($\ell_1$)    &         0.0                  &          0.0    &     0.0       &       0.0        &      0.0\\      &         FCM   ($\ell_1$)  &     23.6  &     39.1  &     20.1  &             27.8  &  25.4\\  &  $K$-means  ($\ell_1$)  &  32.0  &  27.2  &  18.7  &  23.2  &  18.1\\  &  Gen.  Weiszfeld  &  47.1  &  47.4  &  48.0  &  47.4  &  47.3\\\hline
$|\mbox{supp}(F)|=                32$  &             PCM   ($\ell_1$)    &         0.3                  &          0.0    &     0.0       &       0.0        &      0.0\\      &         FCM   ($\ell_1$)  &     28.8  &     39.9  &     36.6  &             42.6  &  38.8\\  &  $K$-means  ($\ell_1$)  &  35.7  &  27.5  &  19.3  &  23.4  &  18.6\\  &  Gen.  Weiszfeld  &  48.1  &  48.0  &  47.9  &  47.8  &  47.9\\\hline
\end{tabular}
\smallskip
\caption{Percentages              of   misclassified  data                   in        Example~\ref{ex:4}}
\end{table}

\end{example}

\begin{example}
\label{ex:5}
In this instance $N_1=200,\,N_2=100$. The results are shown in Table~\ref{table:5}.

\begin{table}
\label{table:5}
\begin{tabular}{|l|l|r|r|r|r|r|}
\hline
$|\mbox{supp}(F)|$                &    Method         &                      $n=10^4$  &                    $n=5\cdot  10^4$  &  $n=10^5$  &  $n=5\cdot  10^5$  &      $n=10^6$  \\
\hline
$|\mbox{supp}(F)|=                8$   &              PCM  ($\ell_1$)   &         0.0                  &          0.0    &  0.0       &  0.0        &      0.0\\  &         FCM  ($\ell_1$)  &  0.0   &  10.0  &  7.2   &  0.4   &  0.4\\   &  $K$-means  ($\ell_1$)  &  4.9   &  13.5  &  0.0   &  4.9   &  14.4\\  &  Gen.  Weiszfeld  &  0.0   &  0.0   &  13.1  &  0.0   &  0.0\\\hline
$|\mbox{supp}(F)|=                16$  &              PCM  ($\ell_1$)   &         0.0                  &          0.0    &  0.0       &  0.0        &      0.0\\  &         FCM  ($\ell_1$)  &  30.8  &  28.0  &  28.3  &  14.8  &  18.3\\  &  $K$-means  ($\ell_1$)  &  22.2  &  31.8  &  18.4  &  17.6  &  32.0\\  &  Gen.  Weiszfeld  &  39.2  &  36.6  &  35.7  &  36.7  &  36.3\\\hline
$|\mbox{supp}(F)|=                24$  &              PCM  ($\ell_1$)   &         0.0                  &          0.0    &  0.0       &  0.0        &      0.0\\  &         FCM  ($\ell_1$)  &  21.0  &  26.1  &  30.3  &  27.5  &  37.6\\  &  $K$-means  ($\ell_1$)  &  32.3  &  36.3  &  22.6  &  18.1  &  35.1\\  &  Gen.  Weiszfeld  &  37.4  &  38.4  &  37.6  &  36.5  &  37.8\\\hline
$|\mbox{supp}(F)|=                32$  &              PCM  ($\ell_1$)   &         1.5                  &          0.0    &  0.0       &  0.0        &      0.0\\  &         FCM  ($\ell_1$)  &  38.0  &  35.0  &  36.5  &  38.5  &  33.5\\  &  $K$-means  ($\ell_1$)  &  35.1  &  36.6  &  23.1  &  18.6  &  35.4\\  &  Gen.  Weiszfeld  &  39.7  &  36.0  &  37.5  &  40.0  &  38.0\\\hline
\end{tabular}
\smallskip
\caption{Percentages              of   misclassified  data                   in        Example~\ref{ex:5}}
\end{table}

\end{example}

In all examples Algorithm~PCM($\ell_1$) was unsurpassed and was the clear winner in Examples~\ref{ex:1}, \ref{ex:2}, \ref{ex:4} and \ref{ex:5}.

\end{document}